\documentclass{amsart}   
\usepackage{graphicx, amsmath, amssymb, amsthm, amscd}
\usepackage{color}
\usepackage{amsfonts}
\usepackage{graphics,array}
\usepackage[all]{xy}
\usepackage{afterpage}
\usepackage{multirow}

\usepackage{enumerate}

\DeclareGraphicsExtensions{.jpg,.pdf,.mps,.png}
\newcommand{\cl}{\operatorname{cl}}

\newtheorem{theorem}{Theorem}[section]
\newtheorem{lemma}[theorem]{Lemma}

\newtheorem{corollary}[theorem]{Corollary}
\newtheorem{problem}[theorem]{Problem}
\newtheorem{conjecture}[theorem]{Conjecture}

\numberwithin{equation}{section}

\begin{document}

\title{On 
the Conjecture of Wood and projective homogeneity}



\author{J P. Boro\'nski}
\author{M. Smith}
\address[J. P. Boro\'nski]{Faculty of Applied Mathematics,
	AGH University of Science and Technology,
	al. Mickiewicza 30,
	30-059 Krak\'ow,
	Poland -- and -- National Supercomputing Centre IT4Innovations, Division of the University of Ostrava,
	Institute for Research and Applications of Fuzzy Modeling,
	30. dubna 22, 70103 Ostrava,
	Czech Republic}
\email{jan.boronski@osu.cz}
\address[M. Smith]{Department of Mathematics and Statistics, Auburn University, Auburn, AL 36849, United States}
\email{smith01@auburn.edu}

\date{}

\begin{abstract}
In 2005 Kawamura and Rambla, independently, constructed a metric counterexample to Wood's Conjecture from 1982. We exhibit a new nonmetric counterexample of a space $\hat L$, such that $C_0(\hat L,\mathbb{C})$ is almost transitive, and show that it is distinct from a nonmetric space whose existence follows from the work of Greim and Rajagopalan in 1997. Up to our knowledge, this is only the third known counterexample to Wood's Conjecture. We also show that, contrary to what was expected, if a one-point compactification of a space $X$ is R.H. Bing's pseudo-circle then $C_0(X,\mathbb{C})$ is not almost transitive, for a generic choice of points. Finally, we point out close relation of these results on Wood's conjecture to a work of Irwin and Solecki on projective Fra\"iss\'e limits and projective homogeneity of the pseudo-arc and, addressing their conjecture, we show that the pseudo-circle is not approximately projectively homogeneous.
\end{abstract}
\maketitle

\section{Introduction}
In 1982 G.V. Wood stated the following conjecture, in the isometric theory of Banach spaces. 

\noindent
\textbf{Wood's Conjecture, \cite{Wo}}. \textit{Suppose $L$ is a locally compact Hausdorff space such that the space of all scalar-valued functions vanishing at infinity $C_0(L,\mathbb{K})$, equipped in the supremum norm, is almost transitive. Then $L$ consists of a single point.}

Wood's conjecture is related to Banach-Mazur rotation problem, which asks whether a separable Banach space with a transitive norm has to be isometric or isomorphic to a Hilbert space. Recall that a Banach space $(Y,||\cdot||)$ is called \textit{almost transitive} if for any $\epsilon>0$ and any $y_1,y_2\in Y$ with $||y_1||=||y_2||=1$ there exists a surjective linear isometry $T$ such that $||Ty_1-y_2||<\epsilon$, and it is called \textit{transitive} if for such $y_1$ and $y_2$ there exists a surjective linear isometry $T$ with $Ty_1=y_2$. As a consequence of Banach-Stone Theorem, Wood's Conjecture is a topological problem. In 1997 Greim and Rajagopalan \cite{GR} proved Wood's Conjecture in the real case; i.e. if $C_0(L,\mathbb{R})$ is almost transitive then $L$ is a singleton. They also showed that the existence of a counterexample in the complex case implies the existence of a nonmetric locally compact Hausdorff space $\tilde L$ such that $C_0(\tilde L,\mathbb{C})$ is transitive. In 2005 Kawamura and Rambla, independently, disproved Wood's conjecture.
\begin{theorem}[Kawamura \cite{KK2}, Rambla \cite{Ra}]\label{metric}
If $X$ is a pseudo-arc and $p\in X$ then for $L=X\setminus\{p\}$ the space $C_0(L,\mathbb{C})$ is almost transitive.
\end{theorem}
The \textit{pseudo-arc} is a homogeneous 1-dimensional compact and connected space first constructed by Knaster \cite{Kn}. It is an object of intense research both in topology \cite{Prajs}, \cite{Lewis} and dynamical systems \cite{Kennedy},\cite{Lewis},\cite{Martensen}, \cite{Minc}. It is hereditarily equivalent and hereditarily indecomposable.  By the work of Greim and Rajagopalan there follows the second, nonmetric counterexample (see also \cite{CF}).
\begin{theorem}[Greim\&Rajagopalan \cite{GR}, Kawamura \cite{KK2}, Rambla \cite{Ra}]\label{Greim}
There exists a locally compact Hausdorff nonmetric space $N$ such that $C_0(N,\mathbb{C})$ is transitive. In addition, $N$ has the following properties.
\begin{itemize}
\item[(a)] $N$ is not first countable,
\item[(b)] $N$ is countably compact,
\item[(c)] countable unions of compact sets are relatively compact,
\item[(d)] $N$ is an uncountable union of nowhere dense components, or connected,
\item[(e)] $N$ is not totally disconnected,
\item[(f)] every compact subset of $N$ is an $F$-space,
\item[(g)] for every open $K_\sigma$ set $U$, $\beta U=\cl U$,
\item[(h)] every open $K_\sigma$ set $U$ contains an open subset $V$ with $\beta V\neq\cl V$,
\item[(i)] every infinite compact subset of $L$ contains a copy of $\beta \mathbb{N}$,
\item[(j)] every non-empty compact $G_\delta$ contains an open set,
\item[(k)] all non-empty open $K_\sigma$ subsets are homeomorphic,
\item[(l)] all nonempty compact $G_\delta$ subsets are homeomorphic.
\end{itemize}
\end{theorem}
In fact, Greim and Rajagopalan showed that the properties (a)-(l) must hold for any space $L$ such that $C_0(L,\mathbb{C})$ is transitive. In this paper we provide a new almost transitive but not transitive example. In 1985 M. Smith constructed a homogeneous, hereditarily equivalent and hereditarily indecomposable nonmetric Hausdorff continuum, since then referred to as \textit{nonmetric pseudo-arc} \cite{Sm}. In the present paper we show that this space may be used to exhibit another counterexample to the conjecture of Wood.
\begin{theorem}\label{M}
Suppose $M$ is the nonmetric pseudo-arc of Smith and $p\in M$. Then for $\tilde{M}=M\setminus\{p\}$ the space $C_0(\tilde{M},\mathbb{C})$ is almost transitive. In addition,
\begin{itemize}
\item[(a)] $\tilde M$ and $M$ are not first countable at any point,
\item[(b)] $\tilde M$ and $M$ are separable,
\item[(c)] $C_0(\tilde M,\mathbb{C})$ is not separable,
\item[(d)] $\tilde M$ and $M$ contain an infinite convergent sequence.
\end{itemize}
\end{theorem}
Note that the property $(d)$ in Theorem \ref{M} ensures that $\tilde M$ is not homeomorphic to $N$, in view of property $(i)$ of Theorem \ref{Greim}, hence $C_0(\tilde M,\mathbb{C})$ is not transitive. Up to our knowledge this is only the third known counterexample to Wood's Conjecture in the complex case. Note also that, by the construction in \cite{GR}, the one-point compactification of the space $N$ in Theorem \ref{Greim} is the so-called ``topological ultracoproduct'' of the pseudo-arc $X$ following $U$, where $U$ is any countably incomplete ultrafilter (see e.g. Chapter 4 in \cite{Aviles}). This is pertinent here because the Banach space ultraproduct construction can be regarded as a rather special case of a projective limit
(in the category of Banach spaces) and since the Gel'fand-Mazur functor
$K\longleftrightarrow C(K)$ intertwines direct and inverse limits (in both directions), the
point is that all previously known counterexamples to Wood's conjecture
are direct limits of pseudo-arcs, while the one in Theorem \ref{M} is an inverse limit. Equivalently, the corresponding space of continuous functions $C_0(\tilde{M})$ can be seen as the direct limit of a system of
homomorphisms $C_0(X_\alpha)\to C_0(X_\beta)$ with $\beta<\alpha<\omega_1$, where $\omega_1$ is the first
uncountable ordinal\footnote{We are grateful to an anonymous referee for bringing these facts on ultraproducts and direct limits to our attention}.

There are known, however, necessary conditions on the space $L$ to have an almost transitive space $C_0(L,\mathbb{C})$.
\begin{theorem} Suppose the one-point compactification $\alpha L$ of a locally compact Hausdorff space $L$ is metrizable and $C_0(L,\mathbb{C})$ is almost positive transitive. Then
\begin{enumerate}
	\item (Wood, \cite{Wo}) $\alpha L$ is connected,
	\item (Cabello S\'anchez, \cite{CF2}) $dim(L)=dim(\alpha L)=1$,
	\item (Rambla, \cite{Ra}) Every subcontinuum of $L$ is hereditarily indecomposable,
	\item (Rambla, \cite{Ra}) $L$ is almost homogeneous; i.e. for any $\epsilon>0$ and $x,y\in L$ there exists a homeomorphism $\tau:L\to L$ such that $d_L(\tau(x),y)<\epsilon$.
	\end{enumerate}
\end{theorem}
In his paper, Rambla also noted that a natural candidate to provide another counterexample would be R.H. Bing's pseudo-circle \cite{Bi}, since it is hereditarily indecomposable and almost homogeneous. The \textit{pseudo-circle} is another important object in topology (see e.g.\cite{KeRo}), and topological dynamics, appearing in dynamical decompositions of the 2-torus \cite{BCJ}, as a Birkhoff-like attractor\cite{BOp}, boundary of a Siegel disk in the complex plane\cite{Che}, as well as minimal attractor of surface diffeomorphisms\cite{Ha} (see also \cite{KY}). In addition to being hereditarily indecomposable and almost homogeneous, the pseudo-circle shares with the pseudo-arc the property that its every proper subcontinuum is homeomorphic to the pseudo-arc, and therefore indeed it is a natural candidate for a counterexample. We show however, that the pseudo-circle does not have the required property, for a generic choice of points.
\begin{theorem}\label{notalmost}
Let $C$ be the pseudo-circle. There is a dense set $D\subseteq C$ such that for any $p\in D$ for $L=C\setminus\{p\}$ the space $C_0(L,\mathbb{C})$ is not almost transitive.
\end{theorem}
In the course of this work we realized that the result of Kawamura and Rambla is closely related to the work of Irwin and Solecki on projective Fra\"iss\'e limits (see also \cite{Kubis}). They proved the following result, which they called approximate projective homogeneity of the pseudo-arc \cite{IS}.
\begin{theorem}
Suppose $X$ is a pseudo-arc and $Y$ is a chainable continuum. Given two surjections $f,g: X \rightarrow Y$ and any $\epsilon>0$ there exists a homeomorphism $\varphi$ of $X$ onto itself so that for all $x \in X$:
\begin{eqnarray*}
|f(x)-g(\varphi(x))| & < & \epsilon.
\end{eqnarray*}
In addition, the pseudo-arc is a unique chainable continuum with such a property.
\end{theorem}
In their paper Irwin and Solecki conjectured that it should be possible, after appropriate adjustments of definitions, to develop a similar theory for pseudo-solenoids, and since then related constructions were developed for the pseudo-arc and other universal objects in topology; see e.g. \cite{BartosovaKwiatkowska}, \cite{Kwiatkowska2}, \cite{Kwiatkowska}, \cite{Oppenheim}. In the planar case a potential analogue of Irwin-Solecki result would mean that the pseudo-circle is the unique approximately projectively homogeneous plane separating circle-like continuum, since it is projectively universal in this class \cite{Fe}, \cite{Ro}. This conjecture is stated also in \cite{Kubis}. We show however, without using projective Fra\"iss\'e limits, that this is not the case.
\begin{theorem}\label{nosurhom}
Let $C$ be the pseudo-circle. There exist surjections $f,g:C\to\mathbb{S}^1$ and $\epsilon>0$ such that for no map $h:C\to C$ we have $d(f, g\circ h) < \epsilon$.
\end{theorem}
Note that for two plane separating circle-like continua, each of them is the intersection of a nested sequence of annuli, and so their first \v Cech cohomology groups are isomorphic to $\mathbb{Z}$. Any map $f$ between the two continua induces a multiplication by a  constant on $\mathbb{Z}$, which we call the \textit{degree}, denoted by $deg(f)$.  As we shall show, the reason the pseudo-circle lacks approximate projective homogeneity is because it $n$-fold covers itself for every positive integer $n$. This is also crucial in the proof of Theorem \ref{notalmost}. Since the set of all self-covering maps is a dense $G_\delta$ in the set of all surjections of the pseudo-circle \cite{BS}, in a sense there is a generic choice of maps that will prevent the pseudo-circle from having the projective homogeneity. In addition, Fearnley (\cite{Fe}, p.510) showed that any plane separating circle-like continuum $K$ admits a degree $1$ surjection from $C$ to $K$. Composing with the self-coverings of the pseudo-circle we get the following two results\footnote{We are grateful to Kazuhiro Kawamura who pointed out that the proofs in this paper along with the result in \cite{Fe} show that Theorem \ref{degree} holds.}
\begin{theorem}\label{degree}
Let $C$ be the pseudo-circle and $K$ a plane separating circle-like continuum. For every positive integer $n$ there exists a degree $n$ surjection $f_n: C\to K$.
\end{theorem}
\begin{theorem}\label{generic}
Let $C$ be the pseudo-circle and $K$ a plane separating circle-like continuum. For any surjection $f:C\to K$ with $deg(f)\neq 0$ there exists a map $g:C\to K$ and $\epsilon>0$ such that for no map $h:C\to C$ we have $d(f, g\circ h) < \epsilon$.
\end{theorem}
We conclude with the following problems.
\begin{problem}(cf. \cite{Ra})
Characterize the Hausdorff locally compact spaces $L$, such that $C_0(L,\mathbb{C})$ is almost transitive. What are such metric spaces?
\end{problem}
\begin{problem}
Is there a natural construction of Smith's nonmetric pseudo-arc as a projective Fra\"iss\'e limit?
\end{problem}
Based on the results in \cite{BS} and \cite{KK3} we state the following conjectures, which may also direct potential results for Fra\"iss\'e limits.
\begin{conjecture}
Suppose $C$ is a pseudo-circle and $Y$ is a planar circle-like continuum. Given two surjections $f,g: C \rightarrow Y$ such that $deg(f)=deg(g)$ and any $\epsilon>0$ there exists a local homeomorphism $\varphi$ of $C$ onto itself so that for all $x \in C$:
\begin{eqnarray*}
|f(x)-g(\varphi(x))| & < & \epsilon.
\end{eqnarray*}
\end{conjecture}
\begin{conjecture}
Suppose $C$ is a pseudo-solenoid and $Y$ is a circle-like continuum with the same first \v Cech cohomology groups. Given two shape equivalences $f,g: C \rightarrow Y$ such that $deg(f)=deg(g)$ and any $\epsilon>0$ there exists a homeomorphism $\varphi$ of $C$ onto itself so that for all $x \in C$:
\begin{eqnarray*}
|f(x)-g(\varphi(x))| & < & \epsilon.
\end{eqnarray*}
\end{conjecture}
Our paper is organized as follows. Section 2 contains preliminaries on almost transitivity, continuum theory, and pseudo-arcs. In Section 3 we prove that $C_0(\tilde{M},\mathbb{C})$ is almost transitive, where $\tilde{M}$ is Smith's nonmetric pseudo-arc $M$ with a point removed. In Section 4 we prove first uncountability and separability of $M$ (and therefore $\tilde{M}$).  In Section 5 we show that $C_0(\tilde{M},\mathbb{C})$ is not separable and $M$ contains convergent infinite sequences. In Section 6 we prove Theorem \ref{notalmost}. Finally, in Section 7 we prove Theorem \ref{nosurhom} and Theorem \ref{generic}, relating our result on the pseudo-circle and Wood's Conjecture to the work of Irwin and Solecki.
\section{Preliminaries}
Let $L$ be a locally compact Hausdorff space. Following \cite{Ra} (see also \cite{GR}), we say that $C_0(L,\mathbb{C})$ is \textit{almost positive transitive} if for any $\epsilon>0$ and $f,g\in C_0(L,\mathbb{C})$ with $||f||=||g||=1$ and such that $f,g\geq 0$ there is an isometry $T:C_0(L,\mathbb{C})\to C_0(L,\mathbb{C})$ with $||Tf-g||<\epsilon$. We say that $C_0(L,\mathbb{C})$ admits \textit{almost polar decompositions} if for every $\epsilon>0$ and every $f\in C_0(L,\mathbb{C})$ there is an isometry $T$ such that $||T|f| -f||<\epsilon$. It is easy to see that $C_0(L,\mathbb{C})$ is almost transitive if and only if it is almost positive transitive and allows almost polar decompositions. As a consequence of the Banach-Stone Theorem in order to study almost positive transitivity one can only deal with isometries $T$ of the form $T(f)=f(h)$, where $h:L\to L$ is a homeomorphism (see e.g.\cite{Ra}).
\begin{lemma}[Rambla]\label{polar} Suppose $L$ is a locally compact Hausdorff space with covering dimension at most $1$. Then $C_0(L,\mathbb{C})$ admits almost polar decompositions.
\end{lemma}
As a consequence of the lemma above for a locally compact Hausdorff space $L$ of covering dimension $1$, in order to prove almost transitivity of $C_0(L,\mathbb{C})$ it is enough to show that it is almost positive transitive.
\begin{theorem}[Kawamura \cite{KK2}, Rambla \cite{Ra}]  Suppose $X$ is a pseudo-arc and $I=[0,1]$ so that $f,g: X \rightarrow I$ are two onto functions and there is a point $P$ so that $f(P) = g(P)=0$. Then if $\epsilon>0$ there exists a homeomorphism $\varphi$ of $X$ onto itself so that $\varphi(P)=P$ and for all $x \in X$:
\begin{eqnarray*}
|f(x)-g(\varphi(x))| & < & \epsilon.
\end{eqnarray*}
In particular, for any $p\in X$ and $L=X\setminus\{p\}$ the space $C_0(L,\mathbb{C})$ is almost positive transitive and hence almost transitive.
\end{theorem}
A \textit{continuum} is a compact and connected space. A continuum is called \textit{chainable} (\textit{circle-like}) if it is the inverse limit of spaces homeomorphic to $[0,1]$ (homeomorphic to $\mathbb{S}^1$). Every chainable (circle-like) continuum in $\mathbb{R}^2$ is a continuous image of the pseudo-arc \cite{Mi} (of the pseudo-circle \cite{Fe}, \cite{Ro}), and all chainable continua can be embedded in $\mathbb{R}^2$.
Now we recall some background theorems from continuum theory by Lewis \cite{WL}.
\begin{theorem}[Lewis] \label{WL1} Suppose that $M$ is a one-dimensional continuum.  Then there exists a one-dimensional continuum $\hat M$ and a continuous decomposition $G$ of $\hat M$ into pseudo-arcs so that the decomposition space $\hat M/G$ is homeomorphic to $M$.  Furthermore if $\pi: \hat M \rightarrow \hat M/G$ is the mapping so that $\pi(x)$ is the unique element of $G$ containing $x$ and $h: \hat M/G \rightarrow \hat M/G$ is a homeomorphism then there exists a homeomorphism $\hat h: \hat M \rightarrow \hat M$ so that $\pi \circ \hat h = h \circ \pi$.
\end{theorem}

Thus we have an open map $\pi$ so that the following diagram commutes:
\begin{center}
\begin{displaymath}
\xymatrix{
\hat M \ar@{<-}[r]^{\hat h }\ar@{->}[d]^{\pi}& \hat M \ar@{->}[d]^{\pi}\\
\hat M / G = M \ar@{<-}[r]^{h}& M = \hat M / G
}
\end{displaymath}
\end{center}

\begin{theorem}[Lewis]\label{WL2} Under the hypothesis of Theorem \ref{WL1},  if for some element $g \in G$ we have $x, y \in g$ then there exists a homeomorphism $\hat h$ so that $\hat h(x) = y$ and $\pi \circ \hat h = \pi$.
\end{theorem}

\begin{theorem}[Lewis]\label{WL3}  Suppose that $X$ is a pseudo-arc and $G$ is a continuous collection of pseudo-arcs filling up $X$ so that for each $x \in X$, $\pi(x)$ is the element of $G$ containing $x$ and $Y = X/G$.  Then $Y$ is a pseudo-arc, and if $h: Y \rightarrow Y$ is a homeomorphism then there exists a homeomorphism $\hat h: X \rightarrow X$ so that $\pi \circ \hat h = h \circ \pi$.
\end{theorem}
\begin{center}
\begin{displaymath}
\xymatrix{
X \ar@{<-}[r]^{\hat h }\ar@{->}[d]^{\pi}& X \ar@{->}[d]^{\pi}\\
Y \ar@{<-}[r]^{h}& Y
}
\end{displaymath}
\end{center}
\begin{theorem} \label{new1} \cite{En}  Suppose that $X = \varprojlim\{X_\alpha, f_\alpha^\beta \}_{\alpha < \beta < \omega_1}$ and $Y = \varprojlim\{Y_\alpha, G_\alpha^\beta \}_{\alpha < \beta < \omega_1}$ are inverse limits and there is an index $\gamma < \omega_1$ and a collection of homeomorphisms $\varphi_\delta$ for $\delta > \gamma$ so that the following diagram commutes:
\begin{center}
\begin{displaymath}
\xymatrix{
X_\gamma \ar@{<-}[r]^{ f_\gamma^\delta }\ar@{->}[d]^{\varphi_\gamma}& X_\delta \ar@{->}[d]^{\varphi_\delta}\\
Y_\gamma \ar@{<-}[r]^{g_\gamma^\delta}& Y_\delta
}
\end{displaymath}
\end{center}
Then there is an induced homeomorphism $\varphi: X \rightarrow Y$ so that the following commute:
\begin{displaymath}
\xymatrix{
X_\gamma \ar@{<-}[r]^{ \pi_\gamma }\ar@{->}[d]^{\varphi_\gamma}& X \ar@{->}[d]^{\varphi}\\
Y_\gamma \ar@{<-}[r]^{\pi_\gamma }& Y
}
\end{displaymath}
\end{theorem}

In 1985 the second author constructed a so-called nonmetric pseudo-arc \cite{Sm}, as follows. Suppose that for each $\alpha < \omega_1$, $X_\alpha$ is a pseudo-arc.  If $\alpha$ is not a limit ordinal then let $G_\alpha$ be a continuous decomposition of $X_\alpha$ so that the decomposition space $X_\alpha/G_\alpha = X_{\alpha-1}$ with corresponding open map $f_{\alpha-1}$.  For each limit ordinal $\gamma$ define $X_\gamma = \varprojlim\{X_\alpha, f_\alpha^\beta\}_{\alpha < \beta < \gamma}$; and $f_\alpha^\beta \circ f_\beta^\gamma = f_\alpha^\gamma$ then $M = \varprojlim\{X_\alpha, f_\alpha^\beta\}_{\alpha < \beta < \omega_1}$ is a nonmetric analogue of the pseudo-arc; $M$ is homogeneous and hereditarily equivalent; the natural projection $\pi_\alpha:M\to X_\alpha$ is a closed map for each $\alpha< \omega_1$ (cf. Claim 3.1.1 in \cite{BS}). In addition, $M$ has covering dimension 1, due to the following result of Katuta.
\begin{lemma}(Katuta, \cite{Katuta})\label{Katuta}
Let $\{X_\alpha,f_\alpha^\beta\}$ be an inverse system over a well-ordered set $\Omega$. If each $X_\alpha$ is a normal space with covering dimension at most $n$, the inverse limit $X$ is $|\Omega|$-paracompact and each projection $\pi_\alpha:X\to X_\alpha$ is a closed map, then the covering dimension of $X$ is at most $n$.
\end{lemma}
\begin{figure}[h]
	\centering
		 \includegraphics[width=0.85\textwidth]{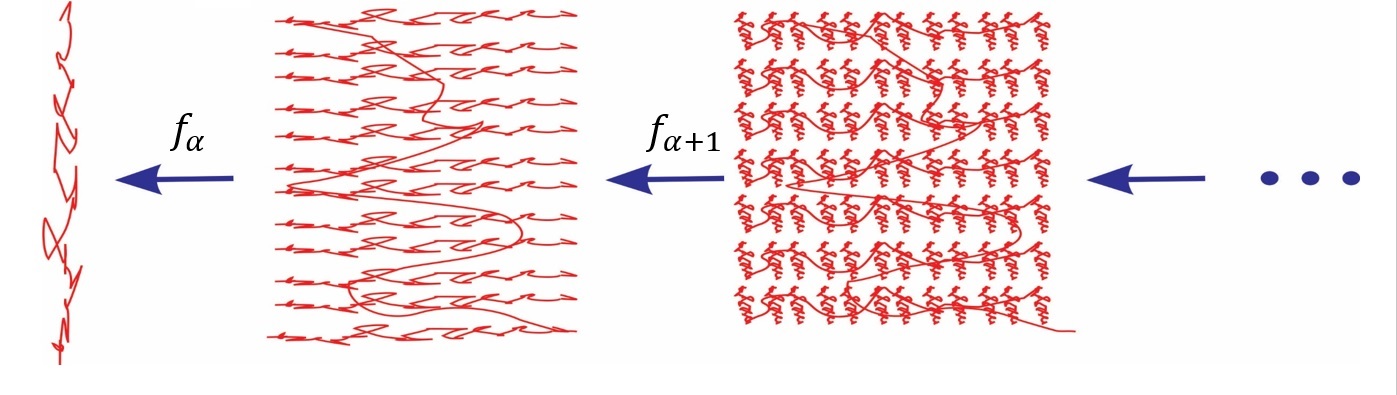}
	\caption{Constructing Smith's nonmetric pseudo-arc}
	\label{fig:michelinvlim}
\end{figure}

\section{Almost positive transitivity of the nonmetric pseudo-arc}

Let $M$ be the nonmetric pseudo-arc with $M=\varprojlim \{X_\alpha, f_\alpha^\beta\}_{\alpha < \beta < \omega_1}$ where for each $\alpha$, $X_\alpha$ is a pseudo-arc and if $G_{\alpha + 1} = \{ f_\alpha^{-1}(x) | x \in X_\alpha \}$ then $G_{\alpha+1}$ is a continuous decomposition of $X_{\alpha + 1}$ with corresponding open monotone map $f_\alpha$ and $X_\alpha = X_{\alpha+1}/G_{\alpha+1}$. The following theorem was proved for the metric pseudo-arc in \cite{KK2} and \cite{Ra}. We will show it holds for the nonmetric pseudo-arc $M$.

\begin{theorem}  Let $P \in M$ and $f,g: M \rightarrow [0,1]$ be maps such that $f(P) = g(P) = 0$, then for each $\epsilon >0$ there is homeomorphism $\varphi: M \rightarrow M$ such that $\varphi(P) = P$ and $|f(t) - g(\varphi(t))| < \epsilon$ for all $t \in M$. In particular, if $L=M\setminus\{P\}$ then $C_0(L,\mathbb{C})$ is almost positive transitive for any $P\in M$.
\end{theorem}
 \begin{proof}

For each $\alpha$ let $d^\alpha$ be a metric for $X_\alpha$.  Suppose that $f$ and $g$ are the functions as given in the hypothesis of the theorem, $P$ is the point and $\epsilon >0$.

Let $\mathcal{R}=\{R_i\}_{i=1}^m$ be a finite collection of open sets covering $M$ so that if $p,q$ lie in the same element of $\mathcal{R}$ then $|f(p) - f(q)| < \min(\frac \epsilon 3, \frac 13)$ and $|g(p) - g(q)| < \min(\frac \epsilon 3, \frac 13)$.  There exists a $\gamma < \omega_1$ and a number $\epsilon_\gamma$ so that (1) if $d^\gamma(p_\gamma, q_\gamma) < \epsilon_\gamma$ then there is an element of $\mathcal{R}$ containing both $p$ and $q$ and (2) for each $x \in X_\gamma$, $\mbox{diam}(f(\pi_\gamma^{-1}(x))< \min(\frac \epsilon 3, \frac 13)$ and $\mbox{diam}(g(\pi_\gamma^{-1}(x))< \min(\frac \epsilon 3, \frac 13)$.
\begin{figure}
	\centering
		 \includegraphics[width=0.85\textwidth]{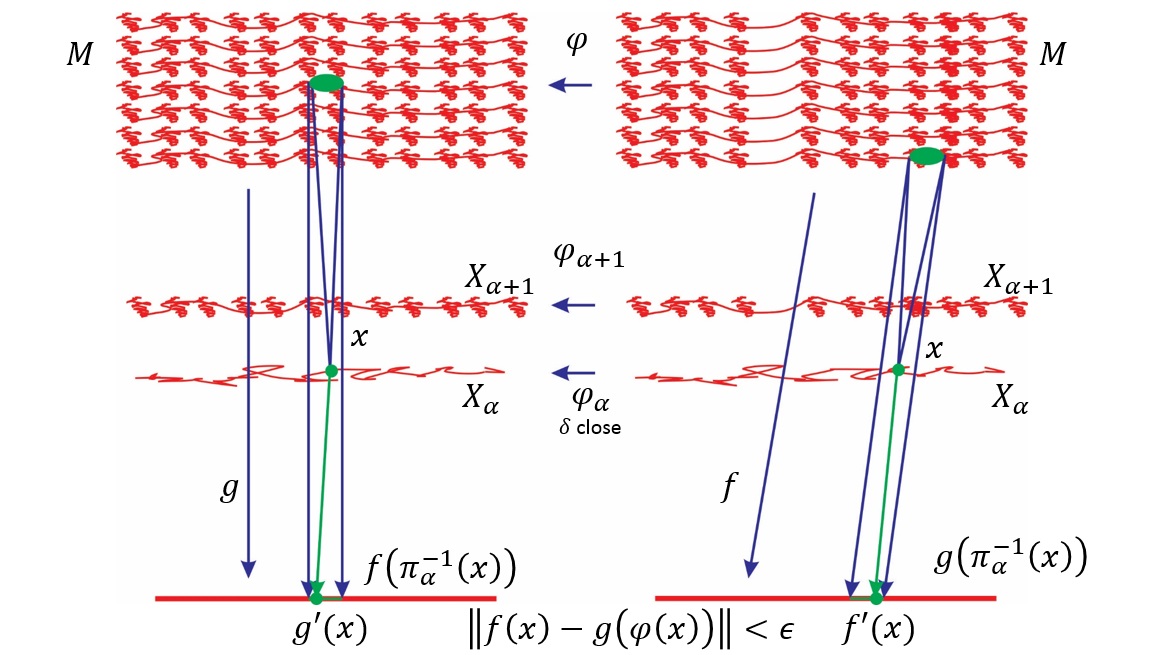}
	\caption{The application of Kawamura-Rambla result at the $\alpha$-level for the auxiliary functions, and an upward inductive bootstraping with repeated instances of Lewis' results in the construction of the homeomorphism $\varphi$.}
	\label{fig:michelinvlim2}
\end{figure}

Observe that the above conditions imply that $\gamma$ is such that $\pi_\gamma^{-1}(x)$ lies in an element of $\mathcal{R}$ for each $x \in X_\gamma$.

Let $F: X_\gamma \rightarrow [0,1]$ be defined by $F(x) = f(\pi_\gamma^{-1}(x))$ and let $G: X_\gamma \rightarrow [0,1]$ be defined by $G(x) = g(\pi_\gamma^{-1}(x))$.  Then for each $x \in X_\gamma$, $F(x)$ is a (possibly degenerate) interval in $[0,1]$ and similarly for $G(x)$.  Let

\begin{eqnarray*}
F(x) = [a_x, b_x] &; & G(x) = [c_x, d_x].
\end{eqnarray*}
Observe that, from the selection of $\gamma$, for each $x \in X_\gamma$, $|b_x - a_x| < \frac 13$;  $|d_x - c_x| < \frac 13$.  And for each $t \in X$ that $a_{t_\gamma} \le f(t) \le b_{t_\gamma}$; $c_{t_\gamma} \le g(t) \le d_{t_\gamma}$.
Define the following function $\hat f: X_\gamma \rightarrow [0,1]$:
$$\hat f(x) = \left \{
\begin{array}{cccc}
\frac{a_x+b_x}2 - \frac{b_x-a_x}2 \Big( \frac{1 - b_x-a_x}{1-b_x}\Big) & \mbox{if } a_x + b_ x \le 1 \\
\frac{a_x+b_x}2 + \frac{b_x-a_x}2 \Big( \frac{ b_x+a_x-1}{a_x}\Big) & \mbox{if } a_x + b_ x > 1
\end{array} \right . $$

Define $\hat g$ similarly with a's replaced by c's and b's replaced by d's.  So the following observations will also be true of $\hat g$.  Note (since $b_x - a_x < \frac 13$) that if $a_x + b_x \le 1$ then $b_x < 1$ so the top part of the function is never undefined and that if $a_x + b_x > 1$ then $a_x >0$ the bottom part is never undefined.  The function has the following properties:

\

\qquad (1) If $a_x = 0$ then $\hat f(x) = 0$;

\qquad (2) If $b_x = 1$ then $\hat f(x) = 1$;

\qquad (3) If $a_x + b_x = 1$ then $\hat f(x) = \frac 12$;

\qquad (4) $\lim_{a_x + b_x \rightarrow 1^+} = \frac 12$;

\qquad (5) $a_x \le \hat f(x) \le b_x$.

\

\noindent Conditions 1-4 can be checked by substitutions.  Condition 5 follows from the fact that, in the definition of the function, less than half the distance between $a_x$ and $b_x$ is added or subtracted from the midpoint of the interval. Conditions 3 and 4 guarantee that $\hat f$ is continuous.  Condition 5 and the choice of $\gamma$ gives us the following inequalities:
\begin{eqnarray*}
|f(t) - \hat f(t_\gamma)| < \frac \epsilon 3 &\ & |f(t) - \hat f(t_\gamma)| < \frac \epsilon 3.
\end{eqnarray*}
Condition (1) implies that $\hat f(P_\gamma) = \hat g(P\gamma) = 0$.  Condition (2) implies that both $\hat f$ and $\hat g$ are onto.

By Theorem  \ref{metric} there exists a homeomorphism $\varphi_\gamma: X_\gamma \rightarrow X_\gamma$ so that $\varphi_\gamma(P_\gamma) = P_\gamma$ and for each $x \in X_\gamma$ ,
\begin{eqnarray*}
|\hat f(x)-\hat g(\varphi_\gamma(x))| & < & \frac \epsilon 3.
\end{eqnarray*}

We consider now the homeomorphism $\varphi_\gamma$.
Using Theorem \ref{WL1} there exists a homeomorphism $\varphi_{\gamma+1}: X_{\gamma+1} \rightarrow X_{\gamma+1}$ so that the following compositions commute:
\begin{center}
\begin{displaymath}
\xymatrix{
X_\gamma \ar@{<-}[r]^{f_\gamma^{\gamma+1} }\ar@{->}[d]^{\varphi_{\gamma}}& X_{\gamma +1} \ar@{->}[d]^{\varphi_{\gamma+1}}\\
X_\gamma \ar@{<-}[r]^{f_\gamma^{\gamma+1}}& X_{\gamma +1}
}
\end{displaymath}
\end{center}
furthermore by Theorem \ref{WL2}, $\varphi_{\gamma+1}$ can be constructed so that $\varphi_{\gamma+1}(P_{\gamma+1}) = P_{\gamma+1}$.
For $\alpha> \gamma$ we define $\varphi_\alpha$ inductively:

Case 1: $\alpha$ is a successor ordinal, $\alpha = \delta + 1$.  Then since $\varphi_\delta: X_\delta \rightarrow X_\delta, \varphi(P_\delta) = P_\delta$, by Lewis' results, Theorem \ref{WL2}, there is a homeomorphism $\varphi_\alpha$ so that the following maps commute:
\begin{center}
\begin{displaymath}
\xymatrix{
X_\delta \ar@{<-}[r]^{f_\delta}\ar@{->}[d]^{\varphi_{\delta}}& X_{\alpha} \ar@{->}[d]^{\varphi_{\alpha}}\\
X_\gamma \ar@{<-}[r]^{f_\delta}& X_{\alpha}
}
\end{displaymath}
\end{center}
Furthermore by Theorem \ref{WL2} $\varphi_\alpha$ can be constructed so that $\varphi_\alpha(P_\alpha) = P_\alpha$.

Case 2: $\alpha$ is a limit ordinal.  Then, by Theorem \ref{new1}, the functions $\{\varphi_\delta\}_{\delta < \alpha}$ induce a homeomorphism $\varphi_\alpha: X_\alpha \rightarrow X_\alpha$ and since $\varphi_\delta(P_\delta) = P_\delta$ for all $\delta < \alpha$ we have $\varphi_\alpha(P_\alpha) = P_\alpha$.

Then the functions $\{\varphi_{\alpha} \}_{\gamma \le \alpha < \omega_1}$ induce a homeomorphism $\varphi: M \rightarrow M$ so that $\varphi(P) = P$.  By construction, for each $x \in M$, we have $(\varphi(x))_\gamma = \varphi_\gamma(x_\gamma)$.

Let $t \in M$.  Then we have the following from our constructions of the functions $\hat f$ and $\hat g$:
\begin{eqnarray*}
|f(t) - \hat f(t_\gamma)| & < & \frac \epsilon 3 \\
|g(\varphi(t)) - \hat g((\varphi(t))_\gamma)| & < & \frac \epsilon 3 \\
|\hat f(t_\gamma) - \hat g(\varphi_\gamma(t_\gamma))| & < & \frac \epsilon 3.
\end{eqnarray*}
Using the triangle inequality we have:
\begin{eqnarray*}
|f(t) - g(\varphi(t))| & < & \epsilon.
\end{eqnarray*}
And since $\varphi(P) = P$, then $\varphi$ is the required function.
\end{proof}
As a consequence of the above theorem, for the nonmetric pseudo-arc $M$ the space $C_0(M \setminus\{p\},\mathbb{C})$ is almost positive transitive for any $p\in M$, and hence by Lemma \ref{polar} and Lemma \ref{Katuta} it is almost transitive.

\section{First uncountability and separability of the nonmetric pseudo-arc.}
Let $M=\varprojlim\{X_\alpha, f_\alpha^\beta\}_{\alpha<\beta<\omega_1}$ be our usual definition of a nonmetric pseudo-arc.  Recall that for each $\alpha$ the collection $G_{\alpha +1} = \{f_\alpha^{-1}(x) | x \in X_\alpha \}$ is a continuous decomposition of $X_{\alpha+1}$ into pseudo-arcs so that the $X_\alpha \cong X_{\alpha+1}/G_{\alpha+1}$.

It has been shown that $M$ is homogeneous and hereditarily equivalent \cite{Sm}.  Suppose that $\alpha< \omega_1$ and $U$ is an open set in $X_\alpha$.  Then a basis element $B(\alpha, U)$ determined by $\alpha$ and $U$ is defined by $B(\alpha, U) = \{ x \in M | x_\alpha \in U\}$;  $\mathcal{B}= \{B(\alpha, U) | \alpha < \omega_1, U \mbox{ is open  in } X_\alpha \}$ is a basis for the topology of $M$.

\begin{theorem}\label{notfirst}  The continuum $M$ is not first countable at each of its points.
\end{theorem}
\begin{proof}  Let $p=\{p_\alpha\}_{\alpha < \omega_1}$ be a point of $M$ and suppose that there is a countable local basis $\{O_n\}_{n=1}^\infty$ at $p$.  Inductively we can construct a countable sequence of ordinals $\{\alpha_n\}_{n=1}^\infty$ and a countable sequence of basis elements $\{B(\alpha_n, U_n)\}_{n=1}^\infty$, with $U_n$ open in $X_n$, so that:
\begin{eqnarray*}
B(\alpha_{n+1},U_{n+1}) & \subset & B(\alpha_n,U_n) \\
B(\alpha_n,U_n) & \subset & \cap_{i=1}^n O_n \\
p_{\alpha_n} & \in & B(\alpha_n,U_n).
\end{eqnarray*}
Since $\{\alpha_n\}_{n=1}^\infty$ is countable there exists $\gamma<\omega_1$ so that $\alpha_n < \gamma$ for all positive integers $n$.  Then there is an element $g_\gamma \in G_\gamma$ so that $p_\gamma \in g_\gamma$.  Let $q_\gamma \in g_\gamma$ be such that $q_\gamma \ne p_\gamma$.  For $\delta < \gamma$ let $q_\delta = p_\delta$.  For $\delta > \gamma$, define inductively $q_\delta$ by:
\begin{eqnarray*}
q_{\delta + 1} & \in & f_\delta^{-1} (q_\delta) .
\end{eqnarray*}
Then for $q=\{q_\delta\}_{\delta < \omega_1}$ we have $q \ne p$ but $q \in B(\alpha_n, U_n)$ for all $n$ so $q \in O_n$ for all $n$.  Therefore $\{O_n\}_{n=1}^\infty$ is not a local basis at $p$.
\begin{figure}[h]
	\centering
\includegraphics[width=0.85\textwidth]{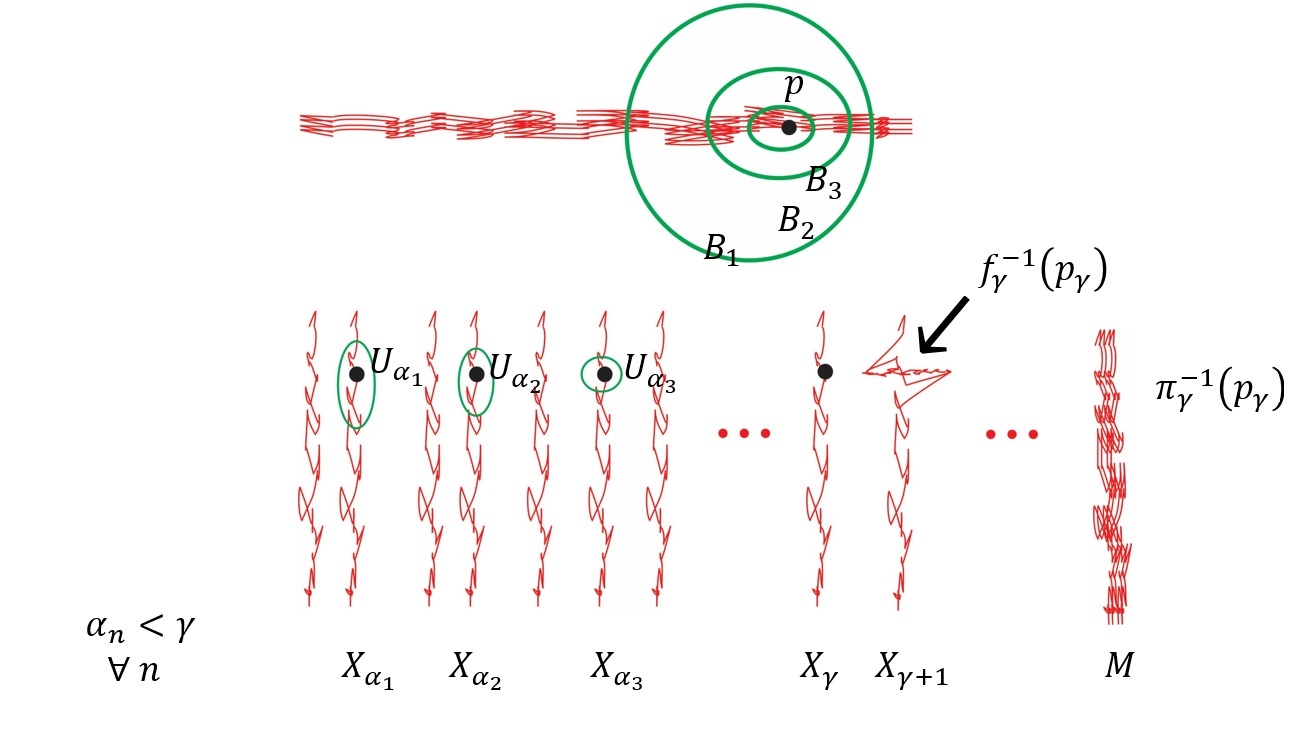}
	\caption{$M$ is not first countable}
	\label{fig:michelinvlim3}
\end{figure}
\end{proof}
\begin{lemma} \label{lem42}
Suppose that $X$ is a metric pseudo-arc, $G$ is a continuous collection of proper subcontinua of $X$ filling up $X$, $Y=X/G$ and $\{g_i\}_{i=1}^\infty$ is a countable set dense in $Y$.  Then there exists a countable set $\{a_i\}_{i=1}^\infty$ so that for each $i$, $a_i \in g_i$ and the set $\{a_i\}_{i=1}^\infty$ is dense in $X$.
\end{lemma}
\begin{proof}
Since $X$ is a compact metric space there is a countable basis $\mathcal{B} = \{B_i\}_{i=1}^\infty$ of open sets for $X$.  Let $G' = \{g_i\}_{i=1}^\infty$.
Let $u_1$ be the first element of $\{g_i\}_{i=1}^\infty$ that intersects $B_1$.  Select a point $a_1 \in B_1 \cap u_1 $.  Inductively let $u_n$ be the first element of $\{g_i\}_{i=1}^\infty - \{u_i\}_{i=1}^{n-1}$ that intersects $B_n$ and let $a_n \in u_n \cap B_n$.  Then we claim that $A= \{a_i\}_{i=1}^\infty$ is the desired set.

\

Claim I. $A$ is dense in $X$.  This follows from the fact that each element of $\mathcal{B}$ contains a point of $A$.

\

Claim II.  Each element of $\{g_i\}_{i=1}^\infty$ contains a single point of $A$. Suppose that some element of $G'$ contains no element of $A$; then let $g_k$ be the first such element.  Then $g_k \notin \{u_i\}_{i=1}^k$; let $B_j$ be the first element of $\{B_i\}_{i=1}^\infty$ that intersects $g_k$ but no element of $\{g_i\}_{i=1}^{k-1}$.  Then $a_j$ would have been selected from $ B_j \cap g_k$.  So each element of $G'$ contains a point of $A$.  No element of $G'$ contains two points of $A$ since at step $k$, all the elements of $G'$ containing one of $\{a_i\}_{i=1}^{k-1}$ have been removed from the elements of $G'$ available for the choice of $u_k$.
\end{proof}

\begin{theorem}
 $M$ is separable.
\end{theorem}
\begin{proof}
Let $\{q_i^1\}_{i=1}^\infty$ be a countable set dense in $X_1$.  For each $\alpha < \omega_1$ we construct inductively a countable sequence of points $\{q_i^\alpha\}_{i=1}^\infty$ dense in $X_\alpha$ so that for each $i$, $f_{\alpha}(q^{\alpha+1}_i) = q^\alpha_i$.

\

Case 1.  $\alpha$ is not a limit ordinal and $\{q_i^{\alpha-1} \}_{i=1}^\infty$ has been constructed.  By Lemma \ref{lem42}, there is a countable collection $A_{\alpha}= \{q_i^{\alpha}\}_{i=1}^\infty$ so that $A_{\alpha}$ is dense in $X_{\alpha}$ and $q_i^{\alpha} \in f_{\alpha-1}^{-1}(q_i^{\alpha-1})$.

\

Case 2.  $\alpha$ is a limit ordinal.  Then for each positive integer $i$ set
$q_i^\alpha = \varprojlim\{\{ q_i^\gamma \}, f_\gamma^\delta\}_{\gamma<\delta<\alpha}$.

\

Then let $q_i = \{q_i^\alpha\}_{\alpha < \omega_1}$; and let $A= \{q_i\}_{i=1}^\infty$.
The set $A$ is countable; to show that $A$ is dense in $M$, select an open set $O$ in $M$.  Then there is a basic open set $B(\alpha, U)$ lying in $O$, where $U$ is open in $X_\alpha$.  Then, by construction  $\{q_i^\alpha\}_{i=1}^\infty$ is dense in $X_\alpha$ so there is an integer $i$ so that $q_i^\alpha \in U$.  Then $q_i \in B(\alpha, U)$ so $q_i \in O$.  Thus $A$ is dense in $M$.
\begin{figure}[h]
	\centering
		 \includegraphics[width=0.99\textwidth]{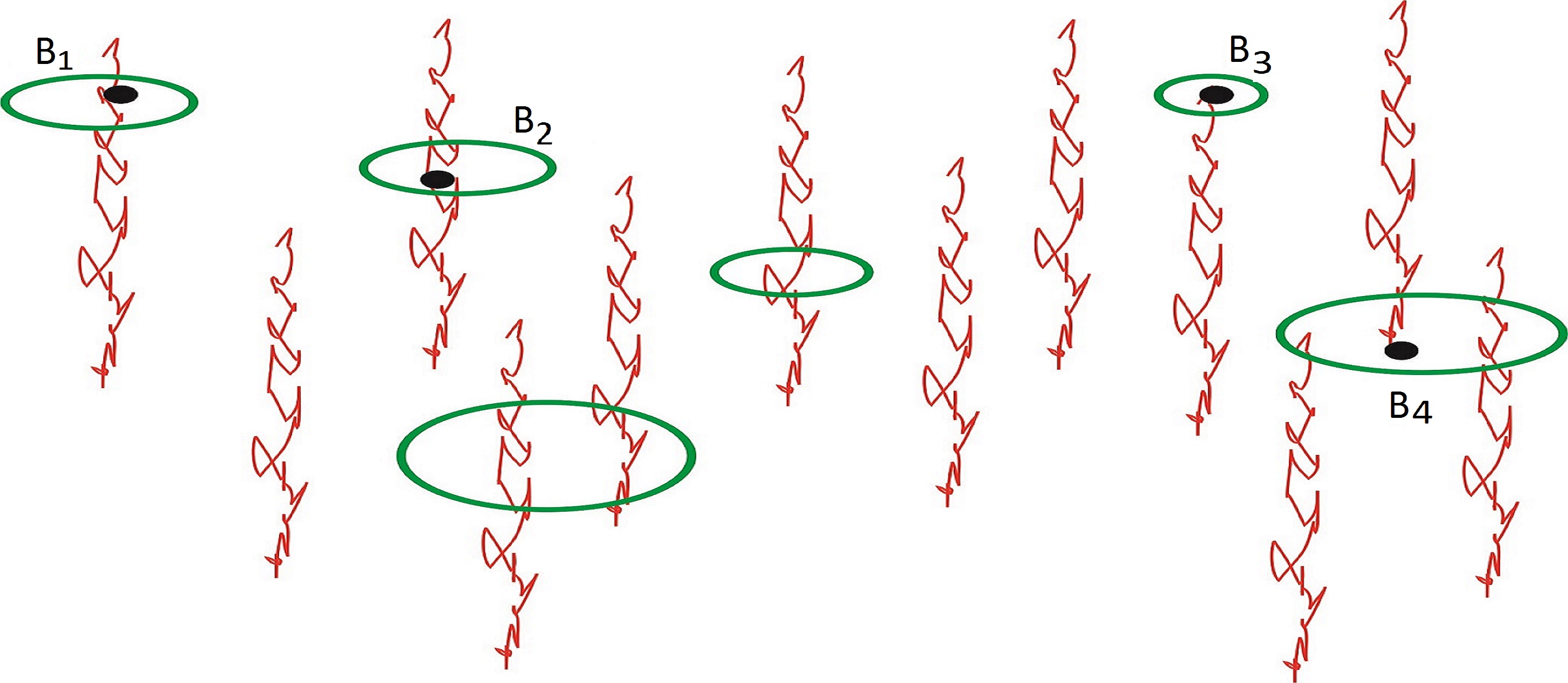}
	\caption{$M$ is separable}
	\label{fig:michelinvlim4}
\end{figure}
\end{proof}

\section{Non-separability of $C_0(\tilde{M},\mathbb{C})$ and existence of convergent infinite sequences in $M$}
Let $M=\varprojlim\{X_\alpha, f_\alpha^\beta\}_{\alpha<\beta<\omega_1}$ be our usual definition of the nonmetric pseudo-arc.  Recall that for each $\alpha$ the collection $G_{\alpha +1} = \{f_\alpha^{-1}(x) | x \in X_\alpha \}$ is a continuous decomposition of $X_{\alpha+1}$ into pseudo-arcs so that the $X_\alpha \cong X_{\alpha+1}/G_{\alpha+1}$. Let $C_0(M)$ denote the set of continuous maps from $M$ into the interval $[0,1]$ with the least upper bound metric $d$. 
Note that neither $C_0(M)$ nor $C_0(M, \mathbb{C})$ is separable, since $M$ and $\tilde{M}$ are nonmetric by Theorem \ref{notfirst}. This justifies Theorem \ref{M} (c). 
Now we prove the following result announced in Theorem \ref{M} (d).
\begin{theorem}
$M$ contains a convergent infinite sequence.
\end{theorem}
\begin{proof}  We construct the sequence by induction.  Let $\{q^1_n\}_{n=1}^\infty$ be a sequence of points of $X_1$ that converge to the sequential limit $q^1$.  Then for each $n$, $g^2_n = f_1^{-1}(q^1_n)$ and $g^2 = f_1^{-1}(q^1)$ are the elements of $G_2$ the continuous decomposition of $X_2$.  So $\{g^2_n\}_{n=1}^\infty$ is a sequence of pseudo-arcs limiting to the pseudo-arc $g^2 \in G_2$.  Then there is a sequence $\{q^2_n\}_{n=1}^\infty$ with $q^2_n \in g^2_n$ and a point $q^2 \in g^2$ which is the sequential limit of the sequence $\{q^2_n\}_{n=1}^\infty$ in $X_2$.

\

Suppose that $\alpha$ is an ordinal and for each $\gamma < \alpha$ we have constructed a sequence  $\{q^\gamma_n\}_{n=1}^\infty$ and a point $q^\gamma$ with $q^\gamma_n \in X_\gamma, q^\gamma \in X_\gamma$ so that $q^\gamma$ is the sequential limit point of the sequence $\{q^\gamma_n\}_{n=1}^\infty$ and for each $\delta < \gamma$ we have, for each $n$, $f_\delta^\gamma(q^\gamma_n) = q^\delta_n$ and $f_\delta^\gamma(q^\gamma) = q^\delta$.

\

Case 1.  $\alpha$ is a limit ordinal.

\noindent Then let $q^\alpha_n = \varprojlim\{ \{q^\gamma_n \}, f_\delta^\gamma\}_{\delta < \gamma < \alpha}$ and $q^\alpha = \varprojlim\{ \{q^\gamma \}, f_\delta^\gamma\}_{\delta < \gamma < \alpha}$.  It follows easily from properties of inverse limits that the sequence $\{q^\alpha_n\}_{n=1}^\infty$ converges to the sequential limit point $q^\alpha$.

\

Case 2.  $\alpha$ is a successor ordinal.

\noindent
Then by the induction hypothesis $\{q^{\alpha -1}_n\}_{n=1}^\infty$ is a sequence of points of $X_{\alpha-1}$ that converge to the sequential limit $q^{\alpha-1}$.  Then for each $n$, $g^\alpha_n = f_{\alpha-1}^{-1}(q^{\alpha-1}_n)$ and $g^\alpha = f_{\alpha-1}^{-1}(q^{\alpha-1})$ are elements of $G_\alpha$ the continuous decomposition of $X_\alpha$.  So $\{g^\alpha_n\}_{n=1}^\infty$ is a sequence of pseudo-arc limiting to the pseudo-arc $g^\alpha \in G_\alpha$.  Then there is a sequence $\{q^\alpha_n\}_{n=1}^\infty$ with $q^\alpha_n \in g^\alpha_n$ and a point $q^\alpha \in g^\alpha$ which is the sequential limit of the sequence $\{q^\alpha_n\}_{n=1}^\infty$ in $X_\alpha$.

\

Let $p_n = \varprojlim\{ \{q^\gamma_n \}, f_\delta^\gamma\}_{\delta < \gamma < \omega_1}$ and $p=\varprojlim\{ \{q^\gamma \}, f_\delta^\gamma\}_{\delta < \gamma < \omega_1}$.  We claim that $\{p_n\}_{n=1}^\infty$ is a sequence of points with sequential limit $p$: For suppose that $B(\alpha, U)$ is a basic open set containing $p$.  Then $q^\alpha \in U \subset X_\alpha$ so there exists an integer $N$ so that if $n > N$ then $q^\alpha_n \in U$; thus, by definition of the inverse limit, $p_n \in B(\alpha, U)$.  So $p$ is the sequential limit of $\{p_n\}_{n=1}^\infty$, and hence there is a convergent sequence in $M$.
\end{proof}

Therefore, by homogeneity, we have the following.
\begin{theorem}  For each point $p \in M$ there is an infinite sequence $\{p_n\}_{n=1}^\infty$ of points of $M$ with sequential limit $p$.
\end{theorem}
\begin{corollary}  There is an infinite sequence $\{p_n\}_{n=1}^\infty$ of points of $M - \{p\}$ with no limit points.
\end{corollary}

\begin{figure}[ht]
	\centering
		 \includegraphics[width=0.6\textwidth]{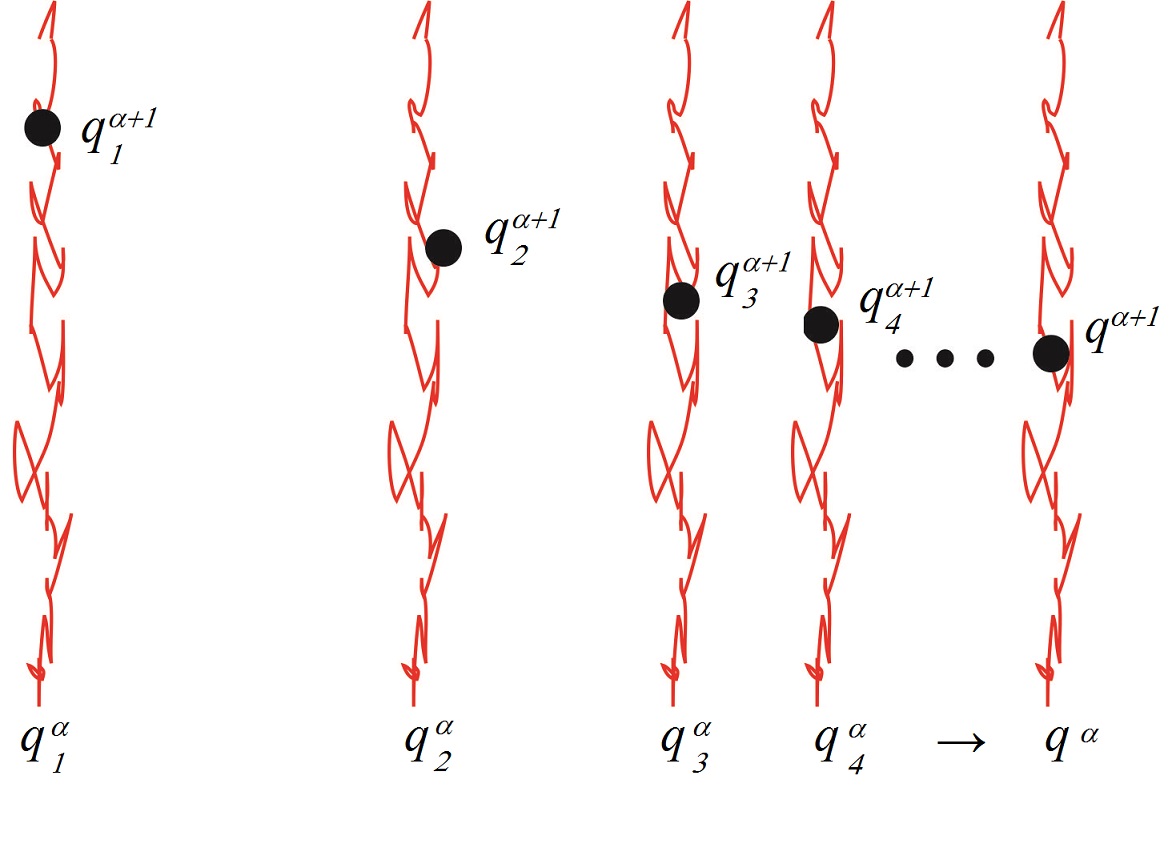}
	\caption{There is a convergent infinite sequence in $M$}
	\label{fig:michelinvlim5}
\end{figure}

\section{The pseudo-circle and the Conjecture of Wood}
We start this section with three background theorems on the pseudo-circle.
\begin{theorem}[Heath \cite{HJ}, Bellamy \cite{Bellamy}]\label{HJ}
Let $C$ be the pseudo-circle. For any integer $n>1$, the pseudo-circle $n$-fold covers itself; i.e. there is an $n$-fold covering map $\phi:C\to C$.
\end{theorem}

\begin{theorem}[Boro\'nski \cite{Bo}]\label{Bo} Suppose that $\phi$ is a degree $r$ pseudo-circle map. Then $\phi$ has at least $|r-1|$ fixed points.
\end{theorem}

\begin{theorem}[Kennedy\&Rogers \cite{KeRo}]\label{KeRo} The pseudocircle has uncountably many orbits under the action of its homeomorphism group. Each of these orbits is the union of uncountably many composants.
\end{theorem}
The following lemma will be helpful, as a direct consequence of the three background theorems on the pseudo-circle.

\begin{lemma}\label{lem}
Let $C$ be the pseudo-circle. For any integer $n>1$, there is a dense set $D\subseteq C$ such that for any $d\in D$ there is an $n$-fold covering map $\phi$ such that $\phi(d)=d$.
\end{lemma}
\begin{proof}
Fix an integer $n>1$. By Theorem \ref{HJ} there is an $n$-fold covering map $\phi':C\to C$. By Theorem \ref{Bo} there is a point $c\in C$ such that $\phi'(c)=c$. By Theorem \ref{KeRo} there is an uncountable collection of composants such that for any point $d$ of any such composant there is a homeomorphism $h_d:C\to C$ such that $h_d(d)=c$. Since any composant is dense in $C$, and a composition of a homeomorphism with an $n$-fold covering map is again an $n$-fold covering map it is enough to set $\phi=h^{-1}_d\circ \phi'\circ h_d$ and notice that $\phi(d)=d$.
\end{proof}
Now we shall prove Theorem \ref{notalmost}. The idea of the proof is as follows. For every integer $n>0$ the pseudo-circle admits an $n$-fold covering map, which extends to a degree $n$ map on an annular neighborhood (which we do not need to use). An $n$-fold covering map composed with a homeomorphism is another $n$-fold covering map, but for $n_1\neq n_2$ no two $n_1$-fold and $n_2$-fold covering maps can be close to each other. So if one measures the distance from a point to its image (as a function from $C$ to the real line) for two such maps, these can never be close.
\begin{proof}(of Theorem \ref{notalmost})
Let $C$ be a pseudo-circle embedded in an annulus $\mathbb{A}$ in such a way so that it separates the two components of the boundary of $\mathbb{A}$. Let $\tau:\hat{\mathbb{A}}\to \mathbb{A}$ be the universal cover of $\mathbb{A}$, where we can assume that $\mathbb{A}=\{(r,\theta)\in \mathbb{R}^2:1\leq r\leq 2,0\leq\theta<2\pi\}$ in polar coordinates, $\hat{\mathbb{A}}=\{(r,\theta)\in\mathbb{R}^2:1\leq r\leq 2\}$, and $\tau(r,\theta)=(r, \theta(\operatorname{mod} 2\pi))$. It will be convenient to use the following metric $d((r,\theta),(r',\theta'))=|r-r'|+|\theta-\theta'|$ for both $\mathbb{A}$ and $\hat{\mathbb{A}}$. To avoid confusion we will indicate in which of the two spaces the distance is taken by writing $d_{\mathbb{A}}$ or $d_{\hat{\mathbb{A}}}$. Note that $\tau$ is a local isometry with respect to the two metrics; i.e. $$d_{\hat{\mathbb{A}}}(\hat x,\hat y)<\epsilon\textrm{ if and only if }d_{\mathbb{A}}(\tau(\hat x),\tau(\hat y))<\epsilon \textrm{ for any } \epsilon<2\pi.$$
Note that, if $k:C\to C$ is a map then it can be continuously extended to a map $\bar k$ of the whole $\mathbb{A}$, as $C$ is a closed subset of the ANR $\mathbb{A}$. Consequently, there exists a lift $\hat k:\tau^{-1}(C)\to \tau^{-1}(C)$ of $k$, namely the restriction of a lift $\bar k:\hat{\mathbb{A}}\to \hat{\mathbb{A}}$ of $H$ to $\tau^{-1}(C)$ i.e. $\tau\circ\hat k=k\circ \tau$. It is known that $\tau^{-1}(C)$ is connected (see \cite{Bo}).

Let $\phi:C\to C$ be a 4-fold covering map, and let $\psi=\phi^2$. Note that $\psi$ is a 16-fold covering map and there exists a point $c\in C$ such that $\phi(c)=\psi(c)=c$ by Theorem 1.1 \cite{Bo}. Let $p=c$. Let $h:C\to C$ be any homeomorphism of $C$ with $h(p)=p=c$. Set $\eta=\psi\circ h$ and note that $\eta$ is a 16-fold covering map and $\eta(c)=c$. Let $f,g:C\to \mathbb{R}$ be defined as follows. Set $f(x)=d_\mathbb{A}(x,\phi(x))$ and $g(x)=d_\mathbb{A}(x,\eta(h(x)))$. Since $f(p)=g(p)=0$ without loss of generality we may assume that $f(C)=g(C)=[0,1]$ (otherwise it suffices to normalize the co-domains).

Fix $\hat c\in\tau^{-1}(c)$. There exist unique lifts  $\hat \phi,\hat \eta:\tau^{-1}(C)\to\tau^{-1}(C)$ such that $\hat \phi(\hat c)=\hat \eta(\hat c)=\hat c$. By the fact that $\phi$ extends to a degree $4$ map and $\eta$ extends to a degree $16$ map on $\mathbb{A}$ these maps induce multiplication by $4$ and $16$ respectively on the level of the fundamental group $\mathbb{Z}$ of $\mathbb{A}$, and hence we have the following stretching in $\hat{\mathbb{A}}$
$$\hat \phi(\hat c+(0,2\pi))=\hat \phi(\hat c)+(0,4\pi),$$
$$\hat \eta(\hat c+(0,2\pi))=\hat \eta(\hat c)+(0,16\pi).$$
Let $\hat f,\hat g:\tau^{-1}(C)\to \mathbb{R}$ be defined by $\hat f(\hat x)=d_{\hat{\mathbb{A}}}(\hat x,\hat \phi(\hat x))$ and $g(\hat x)=d_{\hat{\mathbb{A}}}(\hat x,\hat \eta(\hat x))$. Because $d(\hat f(\hat c),\hat g(\hat c))=0$ and $d(\hat f(\hat c+(0,2\pi)),\hat g(\hat c+(0,2\pi)))=16\pi-4\pi=12\pi$ there exists $\hat w\in\tau^{-1}(C)$ such that $d(\hat f(\hat w),\hat g(\hat w))=\pi$. Since $\tau$ is an isometry on every set of diameter less than $2\pi$, we get $d(f(w),g(w))=\pi$, where $w=\tau(\hat w)$. Hence $d(f,g)\geq \pi$ and the proof is complete in the case $p=c$ with any $\epsilon<\pi$.

The proof is finished by Corollary \ref{lem}, application of which leads to a dense set of points $D$ with the required property.
\end{proof}
\begin{corollary}
There exists a locally compact metric space $L$ such that its one-point compactification is a pseudo-circle and the space $C_0(L,\mathbb{C})$ is not almost transitive.
\end{corollary}

\section{Projective homogeneity and the work of Irwin and Solecki}
In this section we prove Theorems \ref{nosurhom} to \ref{generic}, therefore showing that the work of Irwin and Solecki \cite{IS} does not extend to the pseudo-circle.
\begin{proof}(proof of Theorem \ref{nosurhom})
As in the proof of Theorem \ref{notalmost}, consider $C$ as a subset of $\mathbb{A}$. Let $F:C\to C$ be a 2-fold self-covering map, and $G:C\to C$ a 3-fold self-covering map of $C$. Let $\pi_1:C\to \mathbb{S}^1$ be the radial projection onto a boundary component of $\mathbb{A}$. Set $f=\pi_1\circ F$ and $g=\pi_1\circ G$. Note that since both $C$ and $\mathbb{S}^1$ have the first \v Cech cohomology group isomorphic to $\mathbb{Z}$, there is a well defined degree for $f$ and $g$ equal to $2$ and $3$ respectively. Now let $h:C\to C$ be any map. Then $h$ also has a well defined degree. Since $deg(g\circ h)=deg(g)deg(h)$ we must have $deg(g\circ h)\neq deg(f)$. Now the proof is concluded by proving that no two maps with distinct degrees can be arbitrarily close. Since this is done by a straightforward adaptation of the proof of Theorem \ref{notalmost}, we leave it to the reader.
\end{proof}
\begin{proof}(proof of Theorem \ref{degree})
Let $f:C\to K$ be a surjection with $deg(f)=1$, guaranteed by \cite{Fe}, p.510, and $\phi_n$ be an $n$-fold self-covering of the pseudo-circle. Then $f_n=f\circ\phi_n$ is the desired map from $C$ to $K$ with $deg(f_n)=n$.
\end{proof}

\begin{proof}(proof of Theorem \ref{generic})
Let $f:C\to K$ be a map of degree $n$ and let $g:C\to K$ with $deg(g)=k$ be such that $n$ and $k$ are coprime. The proof is concluded by the arguments in the proof of Theorem \ref{nosurhom}.
\end{proof}
\section{Acknowledgments}
The authors are grateful to K. Kawamura for the many valuable comments and suggestions on the work related to Wood's Conjecture. The authors also express many thanks to E. Glassner, W. Kubi\'s, and D. Barto\v sov\'a for some helpful discussions on Fra\"iss\'e limits. Finally, the authors express gratitude to two anonymous referees for very informative reports with regard to Banach spaces, that greatly helped to improve the paper. 

Early part of this research, that concluded in achieving the results in Section 6, was supported by NPU II project LQ1602 IT4Innovations excellence
in science, with first author's visit to the Department of Mathematics and Statistics, Auburn University in March of 2016. The author is grateful for the hospitality of the department. J. Boro\'nski's subsequent work was supported by National Science Centre, Poland (NCN), grant no. 2015/19/D/ST1/01184 for the project \textit{Homogeneity and Minimality in Compact Spaces}. 

\end{document}